\documentclass[12 pt]{amsart}
\usepackage{amsmath}
\usepackage{amsthm}
\usepackage{amsfonts}
\everymath{\displaystyle}

\newtheorem{theorem}{Theorem}[section]
\newtheorem{lemma}[theorem]{Lemma}

\theoremstyle{definition}

\begin{document}
\title{Subset Powers of Directed Cycles}
\author{Daniel Pragel}
\maketitle

\begin{abstract}
For any directed graph $G$ with vertex set $V,$ the graph $G^{(d)}$ is said to be a \textit{subset power} of $G$ and is defined to have vertex set equal to the set of $d$-element subsets of $V$; in $G^{(d)}$, there is an edge $A \rightarrow B$ if and only if we can label the elements of $A$ and $B$ such that there is an edge in $G$ between each pair of corresponding elements. We determine the complete cycle structure of $C_l^{(d)},$ where $C_l$ is a directed cycle of length $l.$

\end{abstract}

\section{Definitions}
Consider a directed graph $G$ with vertex set $V,$ and let $d$ be any positive integer. As described in \cite{gll}, the graph $G^{(d)}$ is said to be a \textit{subset power} of $G$ and is defined to have vertex set equal to the set of $d$-element subsets of $V,$ such that, for any $d$-element subsets $A$ and $B$ of $V,$ there is an edge $A \rightarrow B$ if and only we can label the elements of $A$ and $B$ by $A = \{a_1,a_2,\dots,a_d\}$ and $B = \{b_1,b_2,\dots,b_d\}$ such that $a_i \rightarrow b_i$ in $G$ for $1 \leq i\leq d.$

\medskip
Let $l$ be a positive integer. $C_l$ is said to be the \textit{directed cycle of length $l$}, which is defined to be the graph with vertex set $\{0,1,2\dots,l-1\}$ and edges $j \rightarrow (j+ 1),$ for $0 \leq j \leq l-2,$ and $(l-1) \rightarrow 0.$ In general, in any directed graph $G,$ for any positive integer $k,$ we say that the ordered $k$-tuple $(v_1,v_2,\dots,v_k)$ of distinct vertices is a \textit{$k$-cycle} if we have $v_j \rightarrow v_{j + 1},$ for $1 \leq j \leq k-1,$ and $v_k \rightarrow v_1.$ Note that $C_l^{(d)}$ will be equal to the disjoint union of $k$-cycles for various $k.$ Let $n(l,d,k)$ be the number of $k$ cycles in $C_l^{(d)}.$ In \cite{gll}, it is shown that, if $l$ is odd, $n(l,2,l) = \frac{l - 1}{2},$ while, if $l$ is even, $n(l,2,l) = \frac{l}{2} - 1$ and $n\left(l,2,\frac{l}{2}\right) = 1.$ We extend these results below.

\section{Results}
First, we prove the following theorem about the existence of $k$-cycles in $C_l^{(d)}$.

\begin{theorem}\label{existence}Let $n(l,d,k)$ be defined as above. Then, $n(l,d,k) \neq 0$ if and only if $k$ divides $l$ and $l$ divides $dk.$

\end{theorem}

\begin{proof}

Let $\oplus$ denote addition modulo $l.$ Suppose that $(A_1,A_2,\dots,A_k)$ is a $k$-cycle in $C_l^{(d)}.$ Then, if $a \in A_1$ we must have $a \oplus s \in A_{1 + s}$ for $1 \leq s \leq k - 1,$ and, since $A_k \rightarrow A_1,$ we must have $a \oplus k \in A_1.$ Therefore, applying this argument inductively, we must have $a \oplus nk \in A_1$ for all positive integers $n.$ If $r \neq 0$ is the remainder when $l$ is divided by $k,$ then this implies that $a \oplus r \in A_1.$ But then, $(A_1,A_2,\dots,A_r)$ forms an $r$-cycle in $C_l^{(d)}$, so $A_{r+1} = A_1,$ contradicting the fact that $(A_1,A_2,\dots,A_k)$ was a $k$-cycle. Therefore, it must be the case that $k$ divides $l,$ which implies that we can partition $A_1$ into disjoint subsets of the form $\{a,a\oplus k,a\oplus2k,\dots,a\oplus(c-1)k\},$ where $c= l/k.$ But this implies that $c$ divides $d,$ and thus $d/c= d/(l/k) = dk/l$ is an integer, which, in turn, implies that $l$ divides $dk.$

\medskip
Conversely, suppose that $k$ divides $l$ and that $l$ divides $dk.$ Let $c = l/k$ and let $t = dk/l.$ Note that, if we let \[A_1 = \{i + jk\,|\, 0 \leq i \leq t-1,0\leq j \leq c-1\}\] and, for $1 \leq i \leq k-1,$ let \[A_{i+1} = \{a \oplus i\,|\, a \in A_1\},\]it is not hard to see that $(A_1,A_2,\dots,A_k)$ is a $k$-cycle in $C_l^{(d)},$ showing that $n(l,d,k) \neq 0.$

\end{proof}

We now prove the following lemma.

\begin{lemma}\label{red}Suppose that $k$ divides $l$ and that $l$ divides $dk.$ Then, we have \[n(l,d,k) = n(k,dk/l,k).\]
\end{lemma}

\begin{proof}
Let $c = l/k$ and let $t = dk/l.$ Let $\mathcal{C}$ be the set of $k$-cycles in $C_l^{(d)}$ and let $\mathcal{C}'$ be the set of $k$-cycles in $C_k^{(t)}.$ Define $\phi: \{0,1,2,\dots,l-1\} \rightarrow \{0,1,2,\dots,k-1\}$ by letting $\phi(j)$ be the remainder when $j$ is divided by $k.$ Suppose that $(A_1,A_2,\dots,A_d) \in \mathcal{C}.$ Let $\oplus$ denote addition modulo $l.$ As shown in the proof of theorem \ref{existence}, we can label the elements of each $A_i$ as follows:\[A_i = \{a_{i,j} \oplus ks\,|\, 1\leq j \leq t, 0 \leq s \leq c-1\}, \] where each $a_{i,j}$ is distinct and is less than or equal to $k-1.$ Then, we have \[\phi(A_i) = \{\phi(a_{i,j} \oplus ks)\,|\,1 \leq j \leq t, 0 \leq s \leq c-1\} = \{a_{i,j}\,|\, 1 \leq j \leq t.\}\]It is not hard to see that $(\phi(A_1),\phi(A_2),\dots,\phi(A_d)) \in \mathcal{C}',$ and, in fact, that $\phi$ induces a bijection from $\mathcal{C}$ to $\mathcal{C}',$ showing that $|\mathcal{C}| = |\mathcal{C}'|$ and, therefore, $n(l,d,k) = n(k,t,k).$
\end{proof}

This leads us to a theorem which gives the complete cycle structure for $C_l^{(d)}.$
\begin{theorem}

Suppose that $k$ divides $l$ and that $l$ divides $dk.$ Let $t = \frac{dk}{l}$ and let \[p_1^{n_1}p_2^{n_2}\dots p_m^{n_m}\] be the prime factorization of the greatest common divisor of $k$ and $t.$ Let $M = \{1,2,\dots,m\}.$ Then, we have \[n(l,d,k) = \frac{1}{t }\sum_{S \subseteq M} (-1)^{|S|} \binom{\frac{k}{\prod_{i \in S} p_i} - 1}{\frac{t}{\prod_{i \in S} p_i} - 1}.\]

\end{theorem}

\begin{proof}

First, note that, by lemma \ref{red}, we have $n(l,d,k) = n(k,t,k).$ Let $(A_1,A_2,\dots,A_k)$ be a $k$-cycle of $C_k^{(t)}.$ Let $\oplus$ denote addition modulo $k.$ The elements of each $A_i$ can be labeled \[A_i = \{a_{i,1},a_{i,2},\dots,a_{i,t}\}\] such that $a_{i+1,j} = a_{i,j} \oplus 1$ for $1 \leq i \leq k,$ $1 \leq j \leq t.$ Thus, it follows that, for fixed $i,$ there must be exactly one $j$ with $a_{i,j} = 0,$ and thus precisely $t$ of the $A_i$ contain 0. 

\medskip
Conversely, consider any $(t-1)$-element subset $B$ of $\{1,2,\dots,k-1\}.$ Let $A_1 = \{0\} \cup B,$ and let \[A_{1 + i} = \{a \oplus i\,|\,a \in A_1\}\]for $i \geq 1.$ We have $A_{k + 1} = A_1.$ Let $r_B$ be the least positive integer $r$ such that $A_{r + 1} = A_1.$ Then, $(A_1,A_2,\dots,A_{r_B})$ is an $r_B$-cycle of $C_k^{(t)}$ and $r_B$ must divide $k.$ Define $\theta$ to be a function from the $(t-1)$-element subsets of $\{1,2,\dots,k-1\}$ to the divisors of $k$ be defined by letting $\theta(B) = r_B.$ Then, for all $a$ dividing $k,$ let $\mathcal{B}_a$ be the set of all $(t-1)$-element subsets $B$ of $\{1,2,\dots,k-1\}$ such that $\theta(B) = k/a.$ It follows that \[n(k,t,k) = \frac{1}{t}|\mathcal{B}_1|.\]Note that our proof of theorem \ref{existence} implies that, given any $a$ which divides $k,$ unless $a$ also divides $t,$ we have $\mathcal{B}_a = \emptyset.$ For all $a$ which divide both $k$ and $t,$ let $\mathcal{D}_a$ be the set of all $t-1$ subsets $B$ of $\{1,2,\dots,k-1\}$ such that $\theta(B)$ divides $k/a.$ Thus, $\mathcal{B}_a \subseteq \mathcal{D}_a.$ 

\medskip
Further, suppose that $a$ and $a'$ each divide both $k$ and $t,$ and that $a'$ divides $a.$ Then, we have $\mathcal{D}_a \subseteq \mathcal{D}_{a'}.$ Therefore, for all $a \neq 1$ such that $a$ divides $k$ and $t,$ there has to be at least one $i$ with $1 \leq i \leq m$ such that $p_i$ divides $a,$ which, in turn, implies that $\mathcal{D}_a \subseteq \mathcal{D}_{p_i}.$

\medskip
Therefore, $\mathcal{D}_1$ is the union of the disjoint sets $\mathcal{B}_1$ and $\textstyle\bigcup_{i \in M} \mathcal{D}_{p_i},$ and we have \[\mathcal{B}_1 = \mathcal{D}_1 \setminus \bigcup_{i \in M} \mathcal{D}_{p_i}.\]Thus, using the principle of inclusion-exclusion (described, in particular, in \cite{wilson}), we see that \[|\mathcal{B}_1| =  |\mathcal{D}_1| + \sum_{S \subseteq M, S \neq \emptyset} (-1)^{|S|}\left|\bigcap_{i \in S} \mathcal{D}_{p_i}\right|.\]Additionally, it is not hard to see that $\textstyle\bigcap_{i \in S} \mathcal{D}_{p_i} = \mathcal{D}_{\prod_{i \in S} p_i},$ and thus \[|\mathcal{B}_1| = \sum_{S \subseteq M} (-1)^{|S|} \left|\mathcal{D}_{\prod_{i \in S} p_i}\right|.\]Let $B$ be a $(t-1)$-element subset of $\{1,2,\dots,k-1\},$ and let $A_i$ be defined as above, with $A_1 = \{0\} \cup B$ and $A_{1 + i} = \{a \oplus i\,|\,a \in A_1\}$ for $i \geq 2.$ Then, $B \in \mathcal{D}_a$ if and only if $A_{k/a+1} = A_1.$Then, define $\phi: \{0,1,2,\dots,k-1\} \rightarrow \{0,1,2,\dots,k/a-1\}$ as in the proof of lemma \ref{red}, by letting $\phi(j)$ be the remainder when $j$ is divided by $k/a.$ Note that $0 \in A_1$ implies $\phi(0) = 0 \in \phi(A_1).$ Further, $\phi(A_1) = \{0\} \cup \phi(B),$ and it is not hard to see that $\phi$ induces a bijection from $\mathcal{D}_a$ to the set of $(t/a - 1)$-element subsets of $\{1,2,\dots,k/a-1\},$ and thus \[|\mathcal{D}_a| = \binom{k/a - 1}{t/a - 1}.\]Therefore, we have \begin{align*}n(l,d,k) &= n(k,t,k) = \frac{1}{t}|\mathcal{B}_1|\\ &= \frac{1}{t}\sum_{S \subseteq M} (-1)^{|S|} \left|\mathcal{D}_{\prod_{i \in S} p_i}\right|\\ &= \frac{1}{t}\sum_{S \subseteq M} (-1)^{|S|}\binom{\frac{k}{\prod_{i \in S}p_i} - 1}{\frac{t}{\prod_{i \in S}p_i} - 1}.\end{align*}

\end{proof}

\end{document}